\documentclass[11pt]{article}
\usepackage{setspace}
\usepackage{amsthm,amsfonts,epsfig, graphics,amsbsy,subfig,amsmath}
\usepackage{xcolor}
\usepackage[colorlinks=true, 
            linkcolor=blue, 
            citecolor=red, 
            urlcolor=blue, 
            filecolor=magenta]{hyperref}
\usepackage{graphicx}
\usepackage[margin=1in]{geometry}
\usepackage{color,float}		
\usepackage{epsfig}

\newcommand{\intRR}{\int\limits_{\mathbb{R}^2} }

\usepackage{makecell}
\usepackage{amssymb}
\usepackage{amsmath}
\usepackage{kotex}
\usepackage{graphicx}
\newtheorem{thr}{Theorem}
\newtheorem{defn}{Definition}

\newtheorem{definition}{Definition}[section]

\newtheorem{remark}[definition]{Remark}

\newcommand{\Lc}{\mathcal{L}}
\newcommand{\Mc}{\mathcal{M}}
\newcommand{\Tc}{\mathcal{T}}
\newcommand{\Rc}{\mathcal{R}}
\newcommand{\vtheta}{\pmb{\theta}}
\newcommand{\vphi}{\pmb{\phi}}
\newcommand{\vf}{\textbf{\textit{f}}}
\newcommand{\vx}{\textbf{\textit{x}}}
\newcommand{\I}{\mathrm{i}}
\newcommand{\wh}{\widehat}

\newcommand{\Nc}{\mathcal{N}}

\newcommand{\Rb}{\mathbb{R}}
\newcommand{\Db}{\mathbb{D}}
\newcommand{\Zb}{\mathbb{Z}}
\graphicspath{{./figures/}}
\title{Gaussian beam Radon transform for tensor fields in $\Rb^2$}
\author{Rohit Kumar Mishra\thanks{Department of Mathematics, Indian Institute of Technology, Gandhinagar, Gujarat, India. \url{rohit.m@iitgn.ac.in}}
\and Souvik Roy\thanks{Department of Mathematics, The University of Texas at Arlington, Arlington, TX, United States of America. \url{souvik.roy@uta.edu}}\and Indrani Zamindar\thanks{Johann Radon Institute for Computational and Applied Mathematics (RICAM), Altenbergerstrasse 69, 4040 Linz, Austria. \url{indrani.zamindar@ricam.oeaw.ac.at}}}
\date{}
\begin{document}
\maketitle
\vspace{-7mm}
\begin{abstract} 
In this article, we introduce and study a set of generalized Gaussian beam Radon transforms (GbRt) acting on tensor fields in $\Rb^2$. The operators considered include longitudinal, transverse, mixed GbRts, along with their integral moments. These operators extend the corresponding notions of the classical generalized Radon transforms for tensor fields. We establish reconstruction results for vector and symmetric 2-tensor fields using appropriate combinations of the defined transforms. This work extends a recent study on the recovery of scalar functions from their GbRt to the recovery of vector and tensor fields from analogously defined generalized GbRts.
\end{abstract}
\vspace{2mm} 
\vspace{2mm} 
\section{Introduction}
The Radon transform, introduced by Johann Radon \cite{Radon1917}, maps a function to its integrals over hyperplanes. In $\mathbb{R}^2$, hyperplanes reduce to lines, and thus the Radon transform coincides with the ray transform, which integrates a function along straight lines. The inversion of the Radon transform can be expressed through filtered back-projection formulas, and its mapping properties are well understood \cite{Helgason1999, Natterer2001}. A natural extension of this framework extends to ray transforms acting on vector and tensor fields. In this setting, one integrates suitable contractions of the field along lines, leading to the theory of vector and tensor tomography. These transforms include longitudinal, transverse, mixed ray transforms, and their moment transforms \cite{Sharafutdinov1994}. A fundamental difference from the scalar case is the presence of nontrivial kernels. For example, the longitudinal ray transform annihilates potential fields and recovers only the solenoidal component. As a result, reconstruction of a general vector or tensor field requires additional data, such as transverse, mixed, or moment ray transforms \cite{denisjuk2006inversion,derevtsov3, krishnan2019momentum,krishnan2020momentum,krishnan2019solenoidal, Louis2024,mishra2020full, Rohit_Suman_2021, Rohit_Suman_2022, PaternainSaloUhlmann2013, Kamran_Otmar_Xray_2tensor, Sharafutdinov1994}. Although the literature contains extensive results on ray transforms of tensor fields, comparatively fewer results are available for Radon transforms. Systematic extensions of the Radon transform from scalar functions to symmetric $m$-tensor fields have been developed in recent works. In particular, longitudinal and transverse Radon transforms for symmetric $m$-tensor fields, together with their weighted transforms, have been introduced, and appropriate combinations of these transforms enable full recovery of the field. The case of vector fields has been studied in \cite{Kunyansky2023, SvetovPolyakova2023}, with further developments for symmetric $m$-tensor fields in $\mathbb{R}^3$ \cite{SvetovPolyakova2024} and, more recently, in $\mathbb{R}^n$ \cite{Abhishek2025, mishra2026radon}.

Computerized tomography (CT) represents a well-established application of the Radon transform in medical imaging. In CT, the measured data correspond to the attenuation of X-ray intensity, which can be accurately modelled as the Radon transform of the underlying attenuation coefficient. The primary objective is to reconstruct this coefficient from its line integrals, typically using the filtered back-projection (FBP) algorithm. This approach relies on approximating X-rays as narrow, straight-line beams, for which the Radon transform provides an effective forward model, and FBP offers an efficient reconstruction technique. In contrast, these assumptions are not directly applicable to optical imaging modalities. Optical beams, such as those produced by lasers, often exhibit non-uniform intensity profiles, and in many laser optics applications, the laser beam is assumed to be a Gaussian beam with a Gaussian intensity profile. In modalities such as optical projection tomography (OPT) and optical coherence tomography (OCT), neglecting the beam's Gaussian nature during reconstruction can lead to image degradation \cite{koskela2019gaussian, ralston2006inverse, ralston2005gaussian, trull2017point, trull2018comparison}. Previous studies \cite{trull2017point, trull2018comparison} have demonstrated that, in OPT, reconstruction methods that explicitly account for the Gaussian beam profile yield improved image quality compared to standard FBP. Specifically, FBP-based reconstructions exhibit tangential blurring and reduced contrast, whereas Gaussian beam-based approaches restore contrast and mitigate such blurring, producing reconstructions that more closely resemble the true image. These findings indicate that, for optical tomographic techniques, incorporating the beam-propagation characteristics, particularly the Gaussian profile, is essential for achieving accurate image reconstruction.

Recently, the authors in \cite{Roy2023} studied the Gaussian beam Radon transform for scalar functions in $\mathbb{R}^2$, where they developed inversion formulas and corresponding numerical reconstruction methods, demonstrating improved performance over the classical filtered back-projection algorithm. However, in many applications, particularly in optical imaging, the underlying quantities of interest are vector or tensor fields arising from electromagnetic wave propagation, where both the Gaussian beam profile and the vectorial nature of the field play a significant role \cite{Levy2019}. Therefore, extending the Gaussian Radon transform to vector and tensor fields is essential for accurately capturing these physical effects and for achieving reliable reconstructions in such settings. In this work, we generalize the Gaussian beam Radon transform to vector and symmetric $2$-tensor fields by introducing longitudinal, transverse, and mixed Gaussian beam transforms, along with their moments. We establish recovery results from appropriate combinations of these transforms. To achieve the goals, we have used the Fourier transform, the Fourier series, and the Mellin transform.

The paper is organized as follows: In Section \ref{Definitions}, we introduce the necessary notations and define the generalized Gaussian beam Radon transforms of our interest. Section \ref{Results} represents the main results along with proofs in subsequent subsections. We finish the paper with acknowledgments in Section \ref{Acknowledgements}.
\vspace{-5mm}
\section{Definitions and notations}\label{Definitions}
\vspace{-4mm}
Throughout the paper, we use bold font letters to denote vectors and regular font letters to denote scalars. Let $\Db \subseteq \Rb^2$ be the unit disk. For $m\ge1$, let $C_c^{\infty}(S^m;\Db)$ denote the space of infinitely differentiable, compactly supported symmetric $m$-tensor fields on $\Db$. This means for $\vf \in C_c^{\infty}(S^m;\Db)$, the components $\displaystyle f_{i_1 \dots i_m}$ (for $1 \leq i_1, \dots , i_m \leq 2)$ of $\vf$ are compactly supported smooth functions defined on $\Db$. We use  $\theta, \phi$ to denote angles, and  $\vtheta,\vphi$ are the corresponding unit vectors, that is, for given $\theta$ and $\phi$ we have $\vtheta=(\cos\theta, \sin\theta) =(\theta_1, \theta_2)$ and $\vphi=(\cos\phi, \sin\phi) = (\phi_1, \phi_2)$. In CT imaging, X-rays are treated as thin, straight beams, which makes the Radon transform a practical and computationally efficient model. Then, using the filtered back-projection (FBP) algorithm on this model, one can effectively reconstruct images. However, optical systems, particularly those based on laser imaging, operate under different physical principles. Laser beams do not have uniform intensity across their profile; instead, they typically show varying intensity distributions. Gaussian beams are particularly common in laser optics because they accurately represent the propagation of coherent light waves through space. 
A Gaussian beam is the fundamental solution of the paraxial wave equation \cite{siegman1986lasers} and describes how a laser beam propagates in free space. Its transverse intensity profile at any propagation distance is Gaussian, and its width changes smoothly along the propagation axis. For a fixed $w_0 \in \mathbb{R}$, the Gaussian beam radius (or beam waist function) $\displaystyle w: \mathbb{R} \rightarrow [w_0, \infty)$ is given by
$$w(u)= w_0\sqrt{1+\left(\frac{u}{z_R}\right)^2} \quad \mbox{and }\quad  z_R= \frac{\pi w_0^2}{\lambda}. $$
Here, $\lambda$ is the wavelength of the beam, $w_0$ is called the minimum beam waist or the smallest transverse radius of the beam, and $z_R$ is the Rayleigh length. Note that the function $w$ is an even function, that is, $w (-u)=w(u)$.

\noindent In this work, we consider the following Gaussian beam kernel or the point spread function (PSF): 
\begin{align}\label{eq:Gaussian Kernel}
    K(x_1,x_2)= \frac{w_0}{w(x_2)}\exp \left(-\frac{x_1^2}{w^2(x_2)}\right).
\end{align}
The classical Radon transform in $\mathbb{R}^2$ is defined as the integral of a function over straight lines. In many laser‑optics applications, monochromatic laser beams do not exhibit X-ray-like uniform intensity; instead, their transverse profiles are typically Gaussian. As a result, each line integral involves convolving the function with a Gaussian beam kernel along the line of integration. With this in mind, we recall the Gaussian‑beam Radon transform of a function $h$, please see \cite{Roy2023} for a more detailed discussion. In the following discussion, $K$ will be the fixed Gaussian kernel defined above in equation \eqref{eq:Gaussian Kernel}.
\begin{defn}
 Let $h\in C_c^{\infty}(\Db)$, $\theta \in [0,2\pi)$, $s\in \Rb$. The \textbf{Gaussian beam Radon transform} of 
    $h$ is defined as 
\begin{align}\label{def:Gauss_beam_scalar}
        \Rc h(s,\theta)= \int_{-\infty}^{\infty} \int_{-\infty}^{\infty} K(s-z, t)\,  h\, (z\vtheta+t\vtheta^\perp)
         \, dz\,dt.
    \end{align}
\end{defn}
\noindent Note that, if we take the kernel $K$ to the Dirac delta supported on straight lines, the definition \eqref{def:Gauss_beam_scalar} reduces to the classical Radon transform. Now, we define a set of generalized Gaussian beam Radon transforms for symmetric $m$-tensor fields in $\Rb^2$, which are our primary objects of study.
\begin{defn}
Let $\vf\in C_c^{\infty}(S^m;\Db)$, $\theta \in [0,2\pi)$, $s\in \Rb$, $k\ge 0$ be an integer and $0\le \ell\le m$. The \textbf{$k^{th}$ moment generalized Gaussian beam Radon transform} of $\vf$ is defined as 
\begin{align}\label{generalized_Gaussian_tensor}
\Mc^k_{\ell}\vf\,(s,\theta)= \int_{-\infty}^{\infty} \int_{-\infty}^{\infty}t^k K(s-z, t)\, \left\langle \vtheta^{\ell}(\vtheta^\perp)^{m-\ell}, \, \vf\, (z\vtheta+t\vtheta^\perp)\right\rangle\, dz\,dt,
\end{align}
\end{defn}
\noindent where the inner product is defined as $\left\langle\vtheta^\ell(\vtheta^{\perp})^{(m-\ell)},\, \vf \right\rangle= \sum_{i_1, \dots, i_m = 1}^2\theta_{i_1}\dots \theta_{i_\ell}\theta^{\perp}_{i_{\ell+1}}\dots \theta^{\perp}_{i_m} f_{i_1\dots i_m}$. \\
We substitute $\vx= z\vtheta+t\vtheta^\perp$, then we have $z=\vtheta\cdot\vx$ and $t=\vtheta^\perp\cdot\vx$. Performing the change of variables, the above definition can be written in the following way:
\begin{defn}
    Let $\vf\in C_c^{\infty}(S^m;\Db)$, $\theta \in [0,2\pi)$, $s\in \Rb$, $k\ge 0$ be an integer and $0\le \ell\le m$. The \textbf{$k^{th}$ moment generalized Gaussian beam Radon transform} of 
    $\vf$ is defined as 
    \begin{align}\label{def:generalized_Gaussian_tensor}
        \Mc^k_{\ell}\vf\,(s,\theta)= \int_{\Rb^2} (\vtheta^\perp \cdot \vx)^k K(s-\vtheta\cdot \vx, \, \vtheta^\perp \cdot \vx)\, \left\langle \vtheta^{\ell}(\vtheta^\perp)^{m-\ell}, \, \vf (\vx)\right\rangle
         \, d\vx.
    \end{align}
\end{defn}
\noindent All these operators are natural generalizations (for the Gaussian kernel) of the well-studied integral transforms, namely longitudinal, transverse, and momentum ray transform. concepts for classical generalized Radon transforms of tensor fields. We will use the abbreviation \textbf{``\textit{GbRt}"} for \textbf{``Gaussian beam Radon transform"}. Depending on the special values of $k$ and $\ell$, we identify the transform defined in equation \eqref{def:generalized_Gaussian_tensor} with different names. In the table below, we list the values of $k$ and $\ell$, along with the corresponding names of $\Mc^k_{\ell}\vf$ that we will use in this article.
\begin{table}[h]
\centering
\resizebox{\textwidth}{!}{\begin{tabular}{|c|c|c|c|}
\hline
&$\ell=0$ & $\ell=m$ & $1\le \ell \le (m-1)$ \\
\hline
$k=0$ &\makecell{longitudinal GbRt \\ ($\Lc\vf$)} &
\makecell{transverse GbRt \\ ($\Tc\vf$)} &
\makecell{mixed GbRt \\ ($\Mc\vf$)} \\
\hline
$k\ge1$ &\makecell{$k^{\text{th}}$ moment longitudinal GbRt\\($\Lc^k\vf$)} &
\makecell{$k^{\text{th}}$ moment transverse GbRt \\ ($\Tc^k\vf$)} &
\makecell{$k^{\text{th}}$ moment mixed GbRt \\ ($\Mc^k\vf$)} \\
\hline
\end{tabular}}
\caption{Integral transforms considered in this article}
\end{table}\\
We conclude this section with a few key notations and definitions that we will use repeatedly. We start by expanding a function $h$ and the known data $\Mc^k_{\ell}\vf$ in the Fourier series with respect to their angular variables  $\phi$ and $\theta$ as follows:
\begin{align}\label{fourier component function}
h(r\vphi)&=\sum_{n\in \Zb} h_n(r)e^{\I n\phi}, \qquad \ \  \text{with} \qquad \ \ \  h_n(r)= \frac{1}{2\pi}\int_0^{2\pi}h(r\vphi) e^{-\I n\phi}\,d\phi.\\\label{fourier component_Lf}
\Mc^k_{\ell}\vf(s,\theta)&=\sum_{n\in \Zb} (\Mc^k_{\ell}\vf)_n(s)e^{\I n\theta}, \quad \text{with} \quad (\Mc^k_{\ell}\vf )_n(s)= \frac{1}{2\pi}\int_0^{2\pi}(\Mc^k_{\ell}\vf)(s,\theta) e^{-\I n\theta}\,d\theta.
 \end{align}
 Next, we recall the Mellin transform and some of its properties:
 \begin{defn}(\cite{Flajolet_1995})
Let $h$ be a complex-valued integrable function that decays at infinity. Then the Mellin transform for $h$ is denoted by $\Nc h$ and is defined by
\begin{align}\label{def:Mellin_transform}
    \Nc h(\rho)= \int_0^{\infty} h(r)r^{\rho -1}\,dr, \quad Re(\rho) > 1.
\end{align}
\end{defn}
\noindent Here we list some properties of the Mellin transform, which are crucial for the discussion that follows:
 \begin{itemize}
     \item[(i)] $\Nc [r^k h(r)](\rho) = \Nc h(\rho+k)$.
     \item [(ii)] If $ \displaystyle (h\times g)(s)= \int_0^{\infty}h(r)g\left(\frac{s}{r}\right)\frac{dr}{r}$ then taking Mellin transform, we get 
     $$\Nc( (h\times g)(s))(\rho)= \Nc h(\rho)\Nc g(\rho).$$
 \end{itemize}
The inverse Mellin transform of a function $h\in C_c^\infty(\Db)$ is given by \cite{Flajolet_1995,titchmarsh1937fourier}:
\begin{align}\label{def:inverse_mellin}
     h(r)= \lim_{T\to\infty} \frac{1}{2\pi \I}  \int_{t-T\I}^{t+T\I} r
^{-\rho}  \Nc h(\rho) \, d\rho.
\end{align}
\begin{remark}
To compactify several upcoming long calculations, we will use the following shorthand notations:  $\mathfrak{c}_{\theta}=\cos{\theta}$ and $\mathfrak{s}_{\theta}=\sin{\theta}$.
\end{remark}
 \section{Main results}\label{Results}
 In this section, we state the main theoretical findings of this article. The theorems provide a method for reconstructing vector fields and symmetric $2$-tensor fields in $\mathbb{R}^2$ from various combinations of their generalized GbRts. We expect a similar approach and analysis can be applied to tensor fields of any order, but the calculations become quite cumbersome, and a compact way to present them is needed. The main challenge for reconstruction is the deconvolution of the defined integrals. The Gaussian structure of the kernel, together with the use of the Mellin transform, facilitates the deconvolution process.
 To achieve our goal, we adopted the techniques introduced by the authors in \cite{Roy2023}. We now present the results; the proofs will be discussed in the subsequent subsections.  
 \begin{thr}\label{vector_rec:Lf+Tf}
     Let $\vf\in C_c^{\infty}(S^1;\Db).$ Then $\vf$ can be recovered from the knowledge of $\Lc\vf$ and $\Tc\vf.$
 \end{thr}
\begin{thr}\label{vector_rec:Lf+L^1f}
     Let $\vf\in C_c^{\infty}(S^1;\Db)$. Then $\vf$ can be recovered from the combination of either $\Lc\vf, \Lc^1\vf$ or $\Tc\vf, \Tc^{1}\vf$.
 \end{thr}
  \begin{thr}\label{2-tenor_rec:Lf+Tf+Mf}
     Let $\vf\in C_c^{\infty}(S^2;\Db).$ Then $\vf$ can be recovered from the knowledge of $\Lc\vf$, $\Tc\vf$ and $\Mc\vf.$
 \end{thr}
  \begin{thr}\label{2-tensor_rec: Lf+:L^1f+L^2f}
     Let $\vf\in C_c^{\infty}(S^2;\Db)$. Then $\vf$ can be recovered from the combination of either $\Lc\vf, \Lc^1\vf, \Lc^2\vf$ or $\Tc\vf, \Tc^1\vf,\Tc^{2}\vf .$
 \end{thr}
 \begin{remark}\label{Rmk:2-tensor_mixed_mom_partial}
   Let $\vf\in C_c^{\infty}(S^2;\Db)$. Then from the knowledge of $\Mc\vf,\Mc^1\vf,\Mc^2\vf$, we cannot fully recover a symmetric $2$-tensor field $\vf$. Therefore, for a full reconstruction of $\vf$, we require additional data, such as $\Lc\vf$, $\Tc\vf$, etc.
 \end{remark}
\subsection{Proof of Theorem \ref{vector_rec:Lf+Tf}}
The objective of this subsection is to recover $\vf=(f_1,f_2)$ from the knowledge of $\Lc\vf$ and $\Tc\vf$. To do this, we first apply the Fourier transform to the known data and then expand it in the Fourier series with respect to their angular variables. Then, the goal will be to find the Fourier coefficients of components of $\vf$ in terms of the Fourier coefficients of known integral transforms.
\begin{proof}
Taking the Fourier transform of $\Lc\vf(s,\theta)$ with respect to the variable $s$, we get
\begin{align}\label{eq:longi_fourier}
    \wh{\Lc\vf} (\sigma,\theta) &= \int_{\Rb} e^{-\I s\sigma}\Lc (s,\theta)\,ds = \int_{\Rb^2} e^{-\I \sigma (\vtheta\cdot \vx)}\wh{K}(\sigma,\vtheta^\perp \cdot \vx)\, \vtheta^\perp \cdot \vf (\vx)\, d\vx.
\end{align}
In a similar way, the Fourier transform of $\Tc\vf(s,
\theta)$ with respect to $s$ gives 
\begin{align}\label{eq:trans_fourier}
    \wh{\Tc\vf}(\sigma,\theta)=\intRR e^{-\I \sigma (\vtheta\cdot \vx)}\wh{K}(\sigma,\vtheta^\perp \cdot \vx)\, \vtheta \cdot \vf (\vx)\, d\vx, 
\end{align}
where $\wh{K}(\sigma,\vtheta^\perp \cdot \vx)= w_0\sqrt{\pi}\exp\left(-\frac{w_0^2\left(1+\left(\frac{\vtheta^\perp \cdot \vx}{z_R}\right)^2\right)\sigma^2}{4}\right)$, since $\wh{\exp({-\frac{a s^2}{2}})}(\sigma) = \sqrt{\frac{2\pi}{a}}\exp({-\frac{\sigma^2}{2a}}).$\\
Multiplying \eqref{eq:longi_fourier} by $\I$ and adding it to \eqref{eq:trans_fourier} gives
\begin{align*}
    &\wh{\Tc\vf}(\sigma,\theta) +\I \wh{\Lc\vf}(\sigma,\theta)=\intRR e^{-\I \sigma (\vtheta\cdot \vx)}\wh{K}(\sigma,\vtheta^\perp \cdot \vx)\, e^{-\I\theta}(f_1(\vx)+\I f_2(\vx))\, d\vx.
\end{align*}
Changing the coordinates from Cartesian to polar by substituting $\vx=r\vphi$, where $\vphi= (\cos\phi, \sin\phi)$, we have
\begin{align*}
     &\wh{\Tc\vf}(\sigma,\theta) +\I \wh{\Lc\vf}(\sigma,\theta)= \int_0^{2\pi}\int_0^{\infty} e^{-\I \sigma r \mathfrak{c}_{(\theta-\phi)}} \wh{K}(\sigma,-r \mathfrak{s}_{(\theta-\phi)})e^{-\I \theta}(f_1(r\vphi)+\I f_2(r\vphi))\,r\,dr\,d\phi.
\end{align*}
 Computing the $(n-1)^{th}$ Fourier coefficients of $\wh{\Tc\vf}+\I \wh{\Lc\vf}$, we have
\begin{align*}
    (\wh{\Tc\vf}+\I \wh{\Lc\vf})_{n-1}(\sigma) &= \frac{1}{2\pi}\int_0^{2\pi}(\wh{\Tc\vf}+\I \wh{\Lc\vf})(\sigma,\theta)e^{-\I(n-1)\theta}\,d\theta\\
    & = \frac{1}{2\pi}\int_0^{2\pi}\int_0^{2\pi} \int_0^\infty e^{-\I \sigma r\mathfrak{c}_{(\theta-\phi)}}e^{-\I n\theta}\wh{K}(\sigma,-r\mathfrak{s}_{(\theta-\phi)})(f_1(r\vphi)+\I f_2(r\vphi))\,r\,dr\,d\phi\,d\theta.
\end{align*}
    Substituting $\theta-\phi=\psi$, we get
    \begin{align*}
    &(\wh{\Tc\vf}+\I \wh{\Lc\vf})_{n-1}(\sigma)\\   &\quad= \frac{1}{2\pi}\int_0^{2\pi}\int_0^{2\pi} \int_0^\infty e^{-\I \sigma r\mathfrak{c}_{\psi}}e^{-\I n(\phi+\psi)}\wh{K}(\sigma,-r\mathfrak{s}_{\psi})(f_1(r\vphi)+\I f_2(r\vphi))\,r\,dr\,d\phi\,d\psi\\
    &\quad= w_0\sqrt{\pi}\int_0^{2\pi} \int_0^\infty e^{-\I \sigma r\mathfrak{c}_{\psi}}e^{-\I n\psi} \exp\left(-\frac{w_0^2\left(1+\left(\frac{r\mathfrak{s}_{\psi}}{z_R}\right)^2\right)\sigma^2}{4}\right)\{(f_1)_n(r)+ \I (f_2)_n(r)\} \,r\,dr\,d\psi\\
    &\quad= w_0\sqrt{\pi} e^{-\frac{w_0^2\sigma^2}{4}}\int_0^\infty\left[\int_0^{2\pi}e^{-\I \sigma r\mathfrak{c}_{\psi}}e^{-\I n\psi}\exp\left(-\frac{w_0^2\left(\frac{r\mathfrak{s}_{\psi}}{z_R}\right)^2\sigma^2}{4}\right)\,d\psi\right]\{(f_1)_n(r)+ \I (f_2)_n(r)\} \,r\,dr\\
    &\quad= 2 w_0\pi^{3/2}e^{-\frac{w_0^2\sigma^2}{4}}\int_0^\infty \{(f_1)_n(r)+ \I (f_2)_n(r)\} g_n(r^{-1}\sigma^{-1}) \,r\,dr,
\end{align*}
 where $ \displaystyle g(t,\psi)= \exp\left( -\frac{w_0^2t^{-2}\mathfrak{s}^2_{\psi}}{4z_{R}^2}-\I t^{-1}\mathfrak{c}_{\psi}\right)$ and $ \displaystyle g_n(t)=\frac{1}{2\pi}\int_{0}^{2\pi}e^{-\I n\psi}g(t,\psi)\,d\psi.$
 \vspace{2mm}\\
The above equation implies
\begin{align}\label{eq:fourier_longi_trans_(n-1)}
    \frac{e^{\frac{w_0^2}{4\sigma^2}}}{2w_0 \pi^{3/2}} (\wh{\Tc\vf}+\I \wh{\Lc\vf})_{n-1}(\sigma^{-1}) &= \int_0^\infty \{(f_1)_n(r)+ \I (f_2)_n(r)\} g_n\left(\frac{\sigma}{r}\right) \,r\,dr \nonumber\\ &=r^2\{(f_1)_n(r)+ \I (f_2)_n(r)\}\times g_n(\sigma).
 \end{align}
Applying the Mellin transform on both sides of the above equation \eqref{eq:fourier_longi_trans_(n-1)} gives 
 \begin{align}\label{eq:Mellin_f1+if2}
 \Nc\left[ \frac{e^{\frac{w_0^2}{4\sigma^2}}}{2w_0 \pi^{3/2}}(\wh{\Tc\vf}+\I \wh{\Lc\vf})_{n-1}(\sigma^{-1})\right](\rho)
&= \Nc\{(f_1)_n+ \I (f_2)_n\}(\rho+2)\Nc g_n(\rho)\nonumber\\
\implies\quad \qquad \qquad \qquad \Nc\{(f_1)_n+ \I (f_2)_n\}(\rho) &= \frac{\Nc\left[ \frac{e^{\frac{w_0^2}{4\sigma^2}}}{2w_0 \pi^{3/2}}(\wh{\Tc\vf}+\I \wh{\Lc\vf})_{n-1}(\sigma^{-1})\right](\rho-2)}{\Nc g_n(\rho-2)}.
 \end{align}
Next, multiplying equation \eqref{eq:longi_fourier} by $\I$ and subtracting it from equation \eqref{eq:trans_fourier}, we have
\begin{align*}
    \wh{\Tc\vf}(\sigma,\theta) -\I \wh{\Lc\vf}(\sigma,\theta)=\intRR e^{-\I \sigma (\vtheta\cdot \vx)}\wh{K}(\sigma,\vtheta^\perp \cdot \vx)\, e^{\I\theta}(f_1(\vx)-\I f_2(\vx))\, d\vx.
\end{align*}
Now we compute the $(n+1)^{th}$ Fourier coefficients of $\wh{\Tc\vf}-\I \wh{\Lc\vf}$. Following a similar procedure to the above, we obtain
\begin{align*}
     (\wh{\Tc\vf}-\I \wh{\Lc\vf})_{n+1}(\sigma)&= 2 w_0 \pi^{3/2}e^{-\frac{w_0^2\sigma^2}{4}}\int_0^\infty \{(f_1)_n(r)-\I (f_2)_n(r)\} g_n(r^{-1}\sigma^{-1}) \,r\,dr\\
     \implies  \frac{e^{\frac{w_0^2}{4\sigma^2}}}{2w_0 \pi^{3/2}}(\wh{\Tc\vf}-\I \wh{\Lc\vf})_{n+1}(\sigma^{-1})&= r^2\{(f_1)_n(r)- \I (f_2)_n(r)\}\times g_n(\sigma).
\end{align*}
By applying the Mellin transform, we have
\begin{align}\label{eq:Mellin_f1-if2}
    \Nc\{(f_1)_n - \I (f_2)_n\}(\rho) &= \frac{\Nc\left[ \frac{e^{\frac{w_0^2}{4\sigma^2}}}{2w_0 \pi^{3/2}}(\wh{\Tc\vf}-\I \wh{\Lc\vf})_{n+1}(\sigma^{-1})\right](\rho-2)}{\Nc g_n(\rho-2)}.
\end{align}
Using the inversion of the Mellin transform  \eqref{def:inverse_mellin}, from equations \eqref{eq:Mellin_f1+if2} and \eqref{eq:Mellin_f1-if2}, we can recover $(f_1)_n+ \I (f_2)_n$ and $(f_1)_n- \I (f_2)_n$ respectively. Thus, we can recover $(f_1)_n $ and $(f_2)_n $ from the knowledge of $\Lc\vf$ and $\Tc\vf$ and therefore we have reconstructed $\vf$. This completes the proof.
\end{proof}
\subsection{Proof of Theorem \ref{vector_rec:Lf+L^1f}}
The goal of the subsection is to prove recovery of a vector field $\vf$ from the knowledge of longitudinal GbRt and its first moment ($\Lc\vf, \Lc^1\vf$) or transverse GbRt and its first moment ($\Tc\vf, \Tc^1\vf$). To avoid repetition, we provide only the reconstruction of  $\vf$ for longitudinal GbRt and its first moment here. Similar techniques can be applied to prove the recovery from transverse GbRt and its first moment.
\begin{proof}
Taking the Fourier transform of $\Lc\vf$ with respect to $s$ variable,  we get (see equation \eqref{eq:longi_fourier})
\begin{align*}
\wh{\Lc\vf} (\sigma,\theta) &= \int_{\Rb^2} e^{-\I \sigma (\vtheta\cdot \vx)}\wh{K}(\sigma,\vtheta^\perp \cdot \vx)\, \vtheta^\perp \cdot \vf (\vx)\, d\vx.
\end{align*}
Changing the coordinates from Cartesian to polar and computing the $n^{th}$ Fourier coefficients of $ \wh{\Lc\vf}$, we get
  \begin{align*}
    \wh{\Lc\vf}_n(\sigma)
    &= \frac{1}{2\pi}\int_0^{2\pi} \int_0^{2\pi} \int_0^{\infty} 
    e^{-\I n\theta} e^{-i\sigma r\mathfrak{c}_{(\theta-\phi)}} \wh{K}(\sigma,-r\mathfrak{s}_{(\theta-\phi)}) (-\mathfrak{s}_{\theta}f_1(r\vphi)+\mathfrak{c}_{\theta}f_2(r\vphi)) \, r \, dr \, d\phi \, d\theta.
\end{align*}
Using $\displaystyle \mathfrak{c}_{\theta}= \frac{e^{\I \theta} + e^{-\I \theta}}{2}$ and $\displaystyle  \mathfrak{s}_{\theta}= \frac{e^{\I \theta} - e^{-\I \theta}}{2 \I}$, we have
\begin{align*}
   2\wh{\Lc\vf}_n(\sigma)
    =& \frac{1}{2\pi}\int_0^{2\pi} \int_0^{2\pi} \int_0^{\infty} 
    e^{-\I (n-1)\theta} e^{-i\sigma r\mathfrak{c}_{(\theta-\phi)}} \wh{K}(\sigma,-r\mathfrak{s}_{(\theta-\phi)}) (f_2+\I f_1)(r\vphi) \, r \, dr \, d\phi \, d\theta 
    \nonumber \\
    & + \frac{1}{2\pi} \int_0^{2\pi} \int_0^{2\pi} \int_0^{\infty} 
    e^{-\I (n+1)\theta} e^{-i\sigma r\mathfrak{c}_{(\theta-\phi)}} \wh{K}(\sigma,-r\mathfrak{s}_{(\theta-\phi)}) (f_2-\I f_1)(r\vphi) \, r \, dr \, d\phi \, d\theta.
    \end{align*}
Put $\theta-\phi=\psi$ in the above equation to get 
 \begin{align*}
   2\wh{\Lc\vf}_n(\sigma)  &= \frac{1}{2\pi} \int_0^{2\pi} \int_0^{2\pi} \int_0^{\infty} 
    e^{-\I (n-1)(\phi+\psi)} e^{-i\sigma r\mathfrak{c}_{\psi}} \wh{K}(\sigma,-r\mathfrak{s}_{\psi}) (f_2+\I f_1)(r\vphi) \, r \, dr \, d\phi \, d\psi 
    \nonumber \\
    &\quad + \frac{1}{2\pi} \int_0^{2\pi} \int_0^{2\pi} \int_0^{\infty} 
    e^{-\I (n+1)(\phi+\psi)} e^{-i\sigma r\mathfrak{c}_{\psi}} \wh{K}(\sigma,-r\mathfrak{s}_{\psi}) (f_2-\I f_1)(r\vphi) \, r \, dr \, d\phi \, d\psi.
    \end{align*}
Using the expression of the Fourier transform of the Gaussian kernel, we have  
\begin{align*}    
2\wh{\Lc\vf}_n(\sigma)&= 2w_0 \pi^{3/2} e^{-\frac{w_0^2\sigma^2}{4}} \Bigg[
\int_0^{\infty} (f_2+\I f_1)_{n-1}(r)\, g_{n-1}(r^{-1}\sigma^{-1})\, r\, dr \\ & \quad \qquad \qquad \qquad \,\,\,\,+ \int_0^{\infty} (f_2-\I f_1)_{n+1}(r)\, g_{n+1}(r^{-1}\sigma^{-1})\, r\, dr \Bigg]
\end{align*}
Hence we have
\begin{align*} \label{eq:Mellin_Lf}
\frac{e^{\frac{w_0^2}{4\sigma^2}}}{w_0 \pi^{3/2}} \wh{\Lc\vf}_n(\sigma^{-1})
&= \int_0^{\infty} \left\{(f_2+\I f_1)_{n-1}(r)\, g_{n-1}\!\left(\frac{\sigma}{r}\right)\ +  (f_2-\I f_1)_{n+1}(r)\, g_{n+1}\!\left(\frac{\sigma}{r}\right)\,\right\} r\, dr
\end{align*}
which implies
\begin{align}
\Nc\left(\frac{e^{\frac{w_0^2}{4\sigma^2}}}{w_0 \pi^{3/2}}\wh{\Lc\vf}_n(\sigma^{-1})\right)(\rho) 
 &= \Nc\left(f_2+\I f_1)_{n-1}\right(\rho+2)\Nc g_{n-1}(\rho)+ \Nc (f_2-\I f_1)_{n+1}(\rho+2)\Nc g_{n+1}(\rho).
\end{align}
Taking the Fourier transform of $\Lc^1\vf$ with respect to $s$, we get
\begin{align*}
        \wh{\Lc^1\vf} (\sigma,\theta) &= \int_{\Rb^2} e^{-\I \sigma (\vtheta\cdot \vx)} (\vtheta^\perp \cdot \vx)\wh{K}(\sigma,\vtheta^\perp \cdot \vx)\, \vtheta^\perp \cdot \vf (\vx)\, d\vx.
    \end{align*}
Going from Cartesian coordinates to polar coordinates gives
\begin{align*}
    \wh{\Lc^1\vf} (\sigma,\theta) &=- \int_0^{2\pi}\int_0^{\infty}e^{-\I \sigma r\mathfrak{c}_{(\theta-\phi)}}r\mathfrak{s}_{(\theta-\phi)} \wh{K}(\sigma,-r\mathfrak{s}_{(\theta-\phi)})(-\mathfrak{s}_{\theta}f_1(r\vphi)+\mathfrak{c}_{\theta}f_2(r\vphi))\,r\,dr\,d\phi.
\end{align*}
Computing the $n^{th}$ Fourier coefficient of $ \wh{\Lc^1\vf}$, we have
  \begin{align*}
      -2\wh{\Lc^1\vf}_n(\sigma) &=\frac{1}{2\pi}\int_0^{2\pi} \int_0^{2\pi}\int_0^{\infty}e^{-\I (n-1)\theta}e^{-i\sigma r\mathfrak{c}_{(\theta-\phi)}} \mathfrak{s}_{(\theta-\phi)}\wh{K}(\sigma,-r\mathfrak{s}_{(\theta-\phi)})(f_2+\I f_1)(r\vphi)\,r^2\,dr\,d\phi\,d\theta \nonumber \\
      &  + \frac{1}{2\pi}\int_0^{2\pi} \int_0^{2\pi}\int_0^{\infty}e^{-\I (n+1)\theta}e^{-i\sigma r\mathfrak{c}_{(\theta-\phi)}} \mathfrak{s}_{(\theta-\phi)}\wh{K}(\sigma,-r\mathfrak{s}_{(\theta-\phi)})(f_2-\I f_1)(r\vphi)\,r^2\,dr\,d\phi\,d\theta \nonumber \\
      &=\frac{1}{2\pi}\int_0^{2\pi} \int_0^{2\pi}\int_0^{\infty}e^{-\I (n-1)(\phi+\psi)}e^{-i\sigma r\mathfrak{c}_\psi} \mathfrak{s}_\psi \wh{K}(\sigma,-r\mathfrak{s}_\psi)(f_2+\I f_1)(r\vphi)\,r^2\,dr\,d\phi\,d\psi \nonumber \\
      &  + \frac{1}{2\pi}\int_0^{2\pi} \int_0^{2\pi}\int_0^{\infty}e^{-\I (n+1)(\phi+\psi)}e^{-i\sigma r\mathfrak{c}_\psi} \mathfrak{s}_\psi\wh{K}(\sigma,-r\mathfrak{s}_\psi)(f_2-\I f_1)(r\vphi)\,r^2\,dr\,d\phi\,d\psi.
  \end{align*}
Substituting $\displaystyle \mathfrak{s}_{\psi}= \frac{e^{\I \psi}- e^{-\I \psi}}{2\I}$ in the above equation, we have
  \begin{align*}
      -4\I\wh{\Lc^1\vf}_n(\sigma) &= \int_0^{2\pi}\int_0^{\infty}\left( e^{-\I((n-2)\psi} -e^{-\I n \psi}\right)e^{-i\sigma r\mathfrak{c}_\psi}\wh{K}(\sigma,-r\mathfrak{s}_\psi)(f_2+\I f_1)_{n-1}(r)\,r^2\,dr\,d\psi\nonumber \\ &\quad + \int_0^{2\pi}\int_0^{\infty}\left( e^{-\I n\psi} -e^{-\I (n+2) \psi}\right)e^{-i\sigma r\mathfrak{c}_\psi}\wh{K}(\sigma,-r\mathfrak{s}_\psi)(f_2-\I f_1)_{n+1} (r)\,r^2\,dr\,d\psi 
      \end{align*}
 Using the expression for the Fourier transform of the Gaussian kernel, we have    
\begin{align*}
  \frac{-2\I e^{\frac{w_0^2}{4 \sigma^2}}}{w_0 \pi^{3/2}}\wh{\Lc^1\vf}_n(\sigma^{-1})  &= \int_0^{\infty}(f_2+\I f_1)_{n-1}(r)\left(g_{n-2}-g_{n}\right)(\frac{\sigma}{r})\,r^2\,dr \\  &\qquad + \int_0^{\infty}(f_2-\I f_1)_{n+1}(r)\left(g_{n}-g_{n+2}\right)(\frac{\sigma}{r})\,r^2\,dr.
  \end{align*}
Taking the Mellin transform of both sides of the above equation, we have
  \begin{align}\label{eq:Mellin_L1f}
   \Nc\left( \frac{-2\I e^{\frac{w_0^2}{4 \sigma^2}}}{w_0 \pi^{3/2}}\wh{\Lc^1\vf}_n(\sigma^{-1})\right)(\rho)  &= \Nc\left((f_2+\I f_1)_{n-1}\right)(\rho+3)\Nc \left(g_{n-2}-g_{n}\right)(\rho)\nonumber \\ &\quad + \Nc\left((f_2-\I f_1)_{n+1}\right)(\rho+3)\Nc \left(g_{n}-g_{n+2}\right)(\rho).
  \end{align}
  From equation \eqref{eq:Mellin_Lf}, we have 
  \begin{align}
     \Nc\left(f_2+\I f_1)_{n-1}\right(\rho)&= \frac{\Nc\left(\frac{e^{\frac{w_0^2}{4\sigma^2}}}{w_0 \pi^{3/2}}\wh{\Lc\vf}_n(\sigma^{-1})\right)(\rho-2)-\Nc (f_2-\I f_1)_{n+1}(\rho)\Nc g_{n+1}(\rho-2)}{\Nc g_{n-1}(\rho-2)} .
  \end{align}
Substituting this into equation \eqref{eq:Mellin_L1f}, we have
  \begin{align*}
&\Nc\left(  \frac{-2\I e^{\frac{w_0^2}{4 \sigma^2}}}{w_0 \pi^{3/2}} \wh{\Lc^1 \vf}_n(\sigma^{-1}) \right)(\rho-3) =  \Nc \left( (f_2 - \I f_1)_{n+1} \right)(\rho)\Nc \left( g_n - g_{n+2} \right)(\rho-3) \nonumber \\
& \quad \quad + \frac{\Nc\left( \frac{e^{\frac{w_0^2}{4\sigma^2}}}{w_0 \pi^{3/2}} \wh{\Lc \vf}_n(\sigma^{-1}) \right)(\rho-2)
    - \Nc (f_2 - \I f_1)_{n+1}(\rho) \Nc g_{n+1}(\rho-2)}{\Nc g_{n-1}(\rho-2)} \Nc\left( g_{n-2} - g_n \right)(\rho-3)
    \end{align*}
    which implies
\begin{align*}
&\Nc g_{n-1}(\rho-2) \Nc \bigg[\frac{-2\I e^{\frac{w_0^2}{4 \sigma^2}}}{w_0 \pi^{3/2}} \wh{\Lc^1 \vf}_n(\sigma^{-1}) \bigg](\rho-3)
- \Nc\left( g_{n-2} - g_n \right)(\rho-3)\Nc \bigg[ \frac{e^{\frac{w_0^2}{4\sigma^2}}}{w_0 \pi^{3/2}} \wh{\Lc \vf}_n(\sigma^{-1}) \bigg](\rho-2) \\
& \hspace{-0.25cm}=\Nc \left( (f_2 - \I f_1)_{n+1} \right)(\rho) \big[ - \Nc g_{n+1}(\rho-2) \Nc\left( g_{n-2} - g_n \right)(\rho-3) + \Nc g_{n-1}(\rho-2)  \Nc \left( g_n - g_{n+2} \right)(\rho-3) \big]
\end{align*}
From the above equation, we have
\begin{align}\label{eq:M(f_2-if_1)_n}
& \Nc \left( (f_2 - \I f_1)_{n} \right)(\rho)\nonumber \\
& \hspace{-0.3cm}= \frac{
    \Nc g_{n-2}(\rho-2) \Nc \bigg[\frac{-2\I e^{\frac{w_0^2}{4 \sigma^2}}}{w_0 \pi^{3/2}} \wh{\Lc^1 \vf}_{n-1}(\sigma^{-1})\bigg](\rho-3)
    - \Nc\left(g_{n-3} - g_{n-1} \right)(\rho-3)\Nc \bigg[\frac{e^{\frac{w_0^2}{4\sigma^2}}}{w_0 \pi^{3/2}} \wh{\Lc \vf}_{n-1}(\sigma^{-1}) \bigg](\rho-2)}{\Nc g_{n-2}(\rho-2) \Nc \left( g_{n-1} - g_{n+1} \right)(\rho-3)
    - \Nc g_{n}(\rho-2) \Nc\left( g_{n-3} - g_{n-1} \right)(\rho-3)
} 
\end{align}
Using the inversion of the Mellin transform \eqref{def:inverse_mellin}, we can recover $ (f_2 - \I f_1)_{n}.$
\vspace{2mm}\\
Then using the knowledge of $\Lc\vf$ and reconstructed $ (f_2 - \I f_1)_{n}$, from equation \eqref{eq:Mellin_Lf}, we have
\begin{align}\label{eq:M(f_2+if_1)_n}
 \Nc\left(f_2+\I f_1)_{n}\right(\rho)&= \frac{ \Nc\left(\frac{e^{\frac{w_0^2}{4\sigma^2}}}{w_0 \pi^{3/2}}\wh{\Lc\vf}_{n+1}(\sigma^{-1})\right)(\rho-2)-\Nc (f_2-\I f_1)_{n+2}(\rho)\Nc g_{n+2}(\rho-2)}{\Nc g_{n}(\rho-2)}.
\end{align}
Again applying Mellin's inversion \eqref{def:inverse_mellin}, we get $ (f_2 + \I f_1)_{n}.$ Therefore,  from \eqref{eq:M(f_2-if_1)_n} and \eqref{eq:M(f_2+if_1)_n}, we can recover $n^{th}$ Fourier coefficients of $f_1$ and $f_2$, hence $\vf$. This completes the proof.
\end{proof}
\subsection{Proof of Theorem \ref{2-tenor_rec:Lf+Tf+Mf}}
In this subsection, we present the reconstruction of a symmetric $2$-tensor field $\vf$ from the knowledge of longitudinal GbRt ($\Lc\vf$), transverse GbRt ($\Tc\vf$) and mixed GbRt ($\Mc\vf$) of $\vf.$
\begin{proof}
     By adding $\Lc\vf$ and $\Tc\vf$, we obtain
     \begin{align*}
         (\Lc\vf + \Tc\vf)(s,\theta)= \int_{\Rb^2} K(\vtheta\cdot \vx -s, \vtheta^\perp \cdot \vx)\,(f_{11}+f_{22})(\vx)\,d\vx.
     \end{align*}
  Taking the Fourier transform with respect to `$s$', we have 
  \begin{align*}
      (\wh{\Lc\vf}+\wh{\Tc\vf})(\sigma,\theta) = \int_{\Rb^2} e^{-\I \sigma (\vtheta\cdot \vx)}\wh{K}(\sigma,\vtheta^\perp \cdot \vx)\, (f_{11}+f_{22})(\vx)\,d\vx.
  \end{align*}
Changing the coordinates from Cartesian to polar by substituting $\vx=r\vphi$ and computing the $n^{th}$ Fourier coefficient, we get 
\begin{align*}
   (\wh{\Lc\vf}+\wh{\Tc\vf})_{n}(\sigma)&=\frac{1}{2\pi}\int_0^{2\pi}(\wh{\Lc\vf}+\wh{\Tc\vf})(\sigma,\theta)e^{-\I n\theta}\,d\theta\\ &=2w_0 \pi^{3/2}e^{-\frac{w_0^2\sigma^2}{4}}\int_0^{\infty}(f_{11}+ f_{22})_n(r)g_n(r^{-1}\sigma^{-1})\,r\,dr,  
\end{align*}
 where $g_n$ is the $n^{th}$ Fourier coefficient of the Fourier transform of  $ g(t,\psi)= \exp\left( -\frac{w_0^2t^{-2}\mathfrak{s}^2_{\psi}}{4z_{R}^2}-\I t^{-1}\mathfrak{c}_{\psi}\right)$.\\
Applying the Mellin transform to the above equation, we get 
\begin{align}\label{Rec:f_11+f_22}
    \Nc\{ (f_{11}+f_{22})_n\}(\rho)=\frac{\Nc\left[\frac{e^{\frac{w_0^2}{4\sigma^2}}}{2w_0 \pi^{3/2}}(\wh{\Lc\vf}+\wh{\Tc\vf})_{n}(\sigma^{-1})\right](\rho-2)}{\Nc g_n(\rho-2)}.
\end{align}
Next, taking Fourier transforms of  $\Lc\vf, \Tc\vf, \Mc\vf$ with respect to `$s$' and adding in the following way, we have
\begin{align}
    (\wh{\Lc\vf}-\wh{\Tc\vf}+2\I\wh{\Mc\vf})(\sigma,\theta)= \int_{\Rb^2} e^{-\I \sigma (\vtheta\cdot \vx)}\wh{K}(\sigma,\vtheta^\perp \cdot \vx)\, e^{2\I\theta}(f_{22}(\vx)- f_{11}(\vx) +2\I f_{12}(\vx))\, d\vx.
\end{align}
Changing the coordinates from Cartesian to polar by substituting $\vx=r\vphi$ and computing the $(n+2)^{th}$ Fourier coefficients of $\wh{\Lc\vf}-\wh{\Tc\vf}+2\I\wh{\Mc\vf}$, we have
\begin{align}
    (\wh{\Lc\vf}-\wh{\Tc\vf}+2\I\wh{\Mc\vf})_{n+2}(\sigma)
    &=\frac{1}{2\pi}\int_0^{2\pi}(\wh{\Lc\vf}-\wh{\Tc\vf}+2\I\wh{\Mc\vf})(\sigma,\theta)e^{-\I(n+2)\theta}\,d\theta \nonumber\\
    &=2w_0 \pi^{3/2}e^{-\frac{w_0^2\sigma^2}{4}}\int_0^{\infty}(f_{22}- f_{11} +2\I f_{12})_n(r)g_n(r^{-1}\sigma^{-1})\,r\,dr.  
\end{align}
Applying the Mellin transform to the above equation, we get 
\begin{align}\label{Rec:f_22-f_11+2if_12}
    \Nc\{ (f_{22}- f_{11} +2\I f_{12})_n\}(\rho)=\frac{\Nc\left[ \frac{e^{\frac{w_0^2}{4\sigma^2}}}{2w_0 \pi^{3/2}}(\wh{\Lc\vf}-\wh{\Tc\vf}+2\I\wh{\Mc\vf})_{n+2}(\sigma^{-1})\right](\rho-2)}{\Nc g_n(\rho-2)}.
\end{align}

\noindent Taking Fourier transforms of  $\Lc\vf, \Tc\vf, \Mc\vf$, with respect to `$s$' and adding in the following way, we have
\begin{align}
    (\wh{\Lc\vf}-\wh{\Tc\vf}-2\I\wh{\Mc\vf})(\sigma,\theta)= \int_{\Rb^2} e^{-\I \sigma (\vtheta\cdot \vx)}\wh{K}(\sigma,\vtheta^\perp \cdot \vx)\, e^{-2\I\theta}(f_{22}(\vx)- f_{11}(\vx) +2\I f_{12}(\vx))\, d\vx.
\end{align}
Changing the coordinates from Cartesian to polar by substituting $\vx=r\vphi$ and computing the $(n-2)^{th}$ fourier coefficients of $\wh{\Lc\vf}-\wh{\Tc\vf}-2\I\wh{\Mc\vf}$, we have
\begin{align}
    (\wh{\Lc\vf}-\wh{\Tc\vf}-2\I\wh{\Mc\vf})_{n-2}(\sigma)&=\frac{1}{2\pi}\int_0^{2\pi}(\wh{\Lc\vf}-\wh{\Tc\vf}+2\I\wh{\Mc\vf})(\sigma,\theta)e^{-\I(n+2)\theta}\,d\theta \nonumber\\
    &=2w_0 \pi^{3/2}e^{-\frac{w_0^2\sigma^2}{4}}\int_0^{\infty}(f_{22}- f_{11} -2\I f_{12})_n(r)g_n(r^{-1}\sigma^{-1})\,r\,dr.  
\end{align}
Applying the Mellin transform to the above equation, we get 
\begin{align}\label{Rec:f_22-f_11-2if_12}
    \Nc\{ (f_{22}- f_{11} -2\I f_{12})_n\}(\rho)=\frac{\Nc\left[ \frac{e^{\frac{w_0^2}{4\sigma^2}}}{2w_0 \pi^{3/2}}(\wh{\Lc\vf}-\wh{\Tc\vf}-2\I\wh{\Mc\vf})_{n-2}(\sigma^{-1})\right](\rho-2)}{\Nc g_n(\rho-2)}.
\end{align}

\noindent Taking the inverse Mellin transform \eqref{def:inverse_mellin} of the equations \eqref{Rec:f_11+f_22},\ \eqref{Rec:f_22-f_11+2if_12},\ \eqref{Rec:f_22-f_11-2if_12}, we can recover the $n^{th}$ Fourier components $(f_{11}+f_{22})_{n}$, $(f_{22}-f_{11}+2\I f
_{12})_n$ and $(f_{22}-f_{11}-2\I f
_{12})_n$ respectively. Combining these we can recover $n^{th}$ Fourier components
$(f_{11})_n$, $(f_{12})_n$, $(f_{22})_n$ of a symmetric 2-tensor field $\vf$. Hence, this completes the proof.
\end{proof}
\subsection{Proof of Theorem \ref{2-tensor_rec: Lf+:L^1f+L^2f}}\label{Sec: Rec_Lf,L^1f,L^2f}
The goal of this subsection is to prove reconstruction of a symmetric $2$-tensor field $\vf$, from the knowledge of either longitudinal GbRt, its first and second moments ($\Lc\vf,\Lc^1\vf,\Lc^2\vf$), or transverse GbRt, its first and second moments ($\Tc\vf,\Tc^1\vf,\Tc^2\vf$). To avoid repetition, we only present recovery from longitudinal data and its moments here. Using a similar approach, one can recover $\vf$ from transverse GbRt and its moments.
\begin{proof}
Taking the Fourier transform of $\Lc\vf$ with respect to `$s$' and changing the coordinates from Cartesian to polar, we have
\begin{align*}
    \wh{\Lc\vf}(\sigma, \theta)&= \int_0^{2\pi}\int_0^{\infty} e^{-\I\sigma r \mathfrak{c}_{(\theta-\phi)}}\wh{K}(\sigma,-r\mathfrak{s}_{(\theta-\phi)})\left(\mathfrak{s}^2_{\theta} f_{11}-2\mathfrak{s}_{\theta}\mathfrak{c}_{\theta} f_{12}+\mathfrak{c}^2_{\theta} f_{22}\right)(r\vphi)\, r\,dr\,d\phi\\
    &= \int_0^{2\pi}\int_0^{\infty} e^{-\I\sigma r \mathfrak{c}_{(\theta-\phi)}}\wh{K}(\sigma,-r\mathfrak{s}_{(\theta-\phi)})\bigg[\frac{1}{2}(f_{22}+f_{11})+\frac{1}{4}(f_{22}-f_{11}+2\I f_{12})e^{2\I \theta} \\
    & \quad \quad \qquad +\frac{1}{4}(f_{22}-f_{11}-2\I f_{12})e^{-2\I \theta}\bigg](r\vphi)\, r\,dr\,d\phi
\end{align*}
Computing the $n^{th}$ Fourier coefficient, we get
\begin{align}
    \wh{\Lc\vf}_n(\sigma) &= \frac{1}{2\pi}\int_0^{2\pi}\int_0^{2\pi}\int_0^{\infty} e^{-\I n \theta}e^{-\I\sigma r \mathfrak{c}_{(\theta-\phi)}}\wh{K}(\sigma,-r\mathfrak{s}_{(\theta-\phi)})\bigg[\frac{1}{4}(f_{22}-f_{11}+2\I f_{12})e^{2\I \theta}\nonumber \\
    & \quad \qquad \quad +\frac{1}{2}(f_{22}+f_{11}) +\frac{1}{4}(f_{22}-f_{11}-2\I f_{12})e^{-2\I \theta}\bigg](r\vphi)\, r\,dr\,d\phi\, d\theta
\end{align}
Substituting $\theta-\phi=\psi$, we get
\begin{align*}
    \wh{\Lc\vf}_n(\sigma)&= \frac{1}{2\pi}\Bigg[\frac{1}{2}\int_0^{2\pi}\int_0^\infty\int_0^{2\pi}e^{-\I n\phi}(f_{11}+f_{22})(r\vphi) e^{-\I n\psi} e^{-\I \sigma r\mathfrak{c}_{\psi}} \wh{K}(\sigma,-r\mathfrak{s}_{\psi})\, r \, d\phi \, dr\, d\psi \nonumber \\
    &+ \frac{1}{4}\int_0^{2\pi}\int_0^\infty\int_0^{2\pi}e^{-\I (n-2)\phi}(f_{22}-f_{11}+2\I f_{12})(r\vphi)\, e^{-\I (n-2)\psi} e^{-\I \sigma r\mathfrak{c}_{\psi}} \wh{K}(\sigma,-r\mathfrak{s}_{\psi})\, r \, d\phi \, dr\, d\psi \nonumber \\
    &+ \frac{1}{4}\int_0^{2\pi}\int_0^\infty\int_0^{2\pi}e^{-\I (n+2)\phi}(f_{22}-f_{11}-2\I f_{12})(r\vphi)\, e^{-\I (n+2)\psi} e^{-\I \sigma r\mathfrak{c}_{\psi}} \wh{K}(\sigma,-r\mathfrak{s}_{\psi})\, r \, d\phi \, dr\, d\psi \Bigg]
    \end{align*}
Using the expression for the Fourier transform of the Gaussian kernel, we have    
\begin{align}
 \wh{\Lc\vf}_n(\sigma)&= 2w_0 \pi^{3/2}e^{-\frac{w_0^2 \sigma^2}{4}}\int_0^{\infty}\bigg[\frac{1}{4}(f_{22}-f_{11}+2\I f_{12})_{n-2}(r)g_{n-2}(r^{-1}\sigma^{-1})\nonumber\\ & \quad \quad+\frac{1}{2}(f_{22}+f_{11})_n(r)g_n(r^{-1}\sigma^{-1})+  \frac{1}{4}(f_{22}-f_{11}-2\I f_{12})_{n+2}(r)g_{n+2}(r^{-1}\sigma^{-1})\bigg]\,r\,dr
\end{align}
Applying the Mellin transform of the above equation, we have
\begin{align}\label{Mellin_Lf}
    \Nc\bigg[\frac{e^{\frac{w_0^2}{4\sigma^2}}}{2w_0 \pi^{3/2}}\wh{\Lc\vf}_n(\sigma^{-1})\bigg](\rho-2)&= \frac{1}{2}\Nc(f_{22}+f_{11})_{n}(\rho)\Nc g_n(\rho-2) \nonumber \\ &+\frac{1}{4}\Nc(f_{22}-f_{11}+2\I f_{12})_{n-2}(\rho)\Nc g_{n-2}(\rho-2)\nonumber \\
    &+\frac{1}{4}\Nc(f_{22}-f_{11}-2\I f_{12})_{n+2}(\rho)\Nc g_{n+2}(\rho-2).
\end{align}
 Taking the Fourier transform of $\Lc^1\vf(s,\theta)$, with respect to `$s$', we get
\begin{align*}
    \wh{\Lc^1\vf}(\sigma, \theta)= \int_{\Rb^2}(\vtheta^\perp \cdot \vx) e^{-\I\sigma(\vtheta\cdot \vx)} \wh{K}(\sigma, \vtheta^\perp \cdot \vx) (\vtheta^\perp)^2\cdot\vf(\vx)\,d\vx.
\end{align*}
Changing the coordinates from Cartesian to polar, we have
\begin{align*}
    \wh{\Lc^1\vf}(\sigma, \theta) &= - \int_0^{2\pi}\int_0^\infty r\mathfrak{s}_{(\theta-\phi)} e^{-\I\sigma r \mathfrak{c}_{(\theta-\phi)}} \wh{K}(\sigma, -r\mathfrak{s}_{(\theta-\phi)})(\vtheta^\perp)^2\cdot\vf(r\vphi)\,r\,dr\, d\phi\\
    &= - \int_0^{2\pi}\int_0^\infty r\mathfrak{s}_{(\theta-\phi)} e^{-\I\sigma r \mathfrak{c}_{(\theta-\phi)}} \wh{K}(\sigma, -r\mathfrak{s}_{(\theta-\phi)})\bigg[\frac{1}{2} (f_{11}+f_{22})(r\vphi)\\& \quad +\frac{1}{4} e^{2\I \theta}\left(f_{22}-f_{11}+2\I f_{12}\right)(r\vphi) +\frac{1}{4} e^{-2\I \theta}\left(f_{22}-f_{11}-2\I f_{12}\right)(r\vphi) \bigg]\, r\,dr\, d\phi.
\end{align*}

Next, computing the $n^{th}$ Fourier coefficient of $\wh{\Lc^1\vf}$, we get
\begin{align*}
-\wh{\Lc^1\vf}_{n}(\sigma)&=\frac{1}{2}\int_0^{2\pi}\int_0^\infty e^{-\I n\psi}\mathfrak{s}_{\psi} e^{-\I \sigma r \mathfrak{c}_{\psi}} \wh{K}(\sigma, -r\mathfrak{s}_{\psi}) (f_{11}+f_{22})_n(r)r^2\,dr\,d\psi\nonumber\\&+ \frac{1}{4}\int_0^{2\pi}\int_0^\infty e^{-\I (n-2)\psi}\mathfrak{s}_{\psi} e^{-\I \sigma r \mathfrak{c}_{\psi}} \wh{K}(\sigma, -rS_{\psi})(f_{22}-f_{11}+2\I f_{12})_{n-2}(r)r^2\,dr\,d\psi\nonumber \\&+ \frac{1}{4}\int_0^{2\pi}\int_0^\infty e^{-\I (n+2)\psi}\mathfrak{s}_{\psi} e^{-\I \sigma r \mathfrak{c}_{\psi}} \wh{K}(\sigma, -r\mathfrak{s}_{\psi})(f_{22}-f_{11}-2\I f_{12})_{n+2}(r)r^2\,dr\,d\psi.
\end{align*}
Substituting $\displaystyle \mathfrak{s}_{\psi}= \frac{e^{\I \psi}- e^{-\I \psi}}{2\I}$ in the above equation and using the expression for the Fourier transform of the Gaussian kernel, we have
\begin{align*}
 &-2\I\wh{\Lc^1\vf}_n(\sigma)= 2w_0 \pi^{3/2}e^{-\frac{w_0^2 \sigma^2}{4}}\int_0^{\infty}\bigg[\frac{1}{4}(f_{22}-f_{11}+2\I f_{12})_{n-2}(r)(g_{n-3}- g_{n-1})(r^{-1}\sigma^{-1})\nonumber\\& +\frac{1}{2}(f_{22}+f_{11})_n(r)(g_{n-1}- g_{n+1})(r^{-1}\sigma^{-1})+  \frac{1}{4}(f_{22}-f_{11}-2\I f_{12})_{n+2}(r)(g_{n+1}-g_{n+3})(r^{-1}\sigma^{-1})\bigg]\,r^2\,dr.
\end{align*}
Taking the Mellin transform of the above equation, we have
\begin{align}\label{Mellin_L^1f}
     \Nc\bigg[-\I \frac{e^{\frac{w_0^2}{4\sigma^2}}}{w_0\pi^{3/2}}\wh{\Lc^1\vf}_n(\sigma^{-1})\bigg](\rho-3)&= \frac{1}{4}\Nc(f_{22}-f_{11}+2\I f_{12})_{n-2}(\rho)\Nc(g_{n-3}- g_{n-1})(\rho-3)\nonumber\\
     &\quad  + \frac{1}{2}\Nc(f_{22}+f_{11})_n(\rho)\Nc(g_{n-1}- g_{n+1})(\rho-3)\nonumber\\
     &\quad+\frac{1}{4}\Nc(f_{22}-f_{11}-2\I f_{12})_{n+2}(\rho)\Nc(g_{n+1}-g_{n+3})(\rho-3).
\end{align}
Taking the Fourier transform of $\Lc^2\vf(s,\theta)$ with respect to `$s$', we get
\begin{align*}
    \wh{\Lc^2\vf}(\sigma, \theta)= \int_{\Rb^2}(\vtheta^\perp \cdot \vx)^2 e^{-\I\sigma(\vtheta\cdot \vx)} \wh{K}(\sigma, \vtheta^\perp \cdot \vx) (\vtheta^\perp)^2\cdot\vf(\vx)\,d\vx.
\end{align*}
Changing the coordinates from Cartesian to polar, we have
\begin{align*}
    \wh{\Lc^2\vf}(\sigma, \theta)&=  \int_0^{2\pi}\int_0^\infty r^2\mathfrak{s}^2_{(\theta-\phi)} e^{-\I\sigma r \mathfrak{c}_{(\theta-\phi)}} \wh{K}(\sigma, -r\mathfrak{s}_{(\theta-\phi)})\bigg[\frac{1}{2} (f_{11}+f_{22})(r\vphi)\\& \qquad +\frac{1}{4} e^{2\I \theta}\left(f_{22}-f_{11}+2\I f_{12}\right)(r\vphi) +\frac{1}{4}e^{-2\I \theta}\left(f_{22}-f_{11}-2\I f_{12}\right)(r\vphi) \bigg]\, r\,dr\, d\phi.
\end{align*}
Computing the $n^{th}$ Fourier coefficient of $\wh{\Lc^2\vf}$, we get
\begin{align*}
    \wh{\Lc^2\vf}_n(\sigma)&= \frac{1}{2}\int_0^{2\pi}\int_0^{\infty} e^{-\I n\psi}\mathfrak{s}^2_{\psi}e^{-\I \sigma r \mathfrak{c}_{\psi}}\wh{K}(\sigma, -r\mathfrak{c}_{\psi}) (f_{11}+f_{22})_n(r)\,r^3\,dr\, d\psi \nonumber \\&+ \frac{1}{4}\int_0^{2\pi}\int_0^{\infty} e^{-\I (n-2)\psi}\mathfrak{s}^2_{\psi}e^{-\I \sigma r \mathfrak{c}_{\psi}}\wh{K}(\sigma, -r\mathfrak{s}_{\psi}) (f_{22}-f_{11}+2\I f_{12})_{n-2}(r)\,r^3\,dr\, d\psi \nonumber \\&+ \frac{1}{4}\int_0^{2\pi}\int_0^{\infty} e^{-\I (n+2)\psi}\mathfrak{s}^2_{\psi}e^{-\I \sigma r \mathfrak{c}_{\psi}}\wh{K}(\sigma, -r\mathfrak{s}_{\psi}) (f_{22}-f_{11}-2\I f_{12})_{n+2}(r)\,r^3\,dr\, d\psi. 
    \end{align*}
Substituting $\displaystyle \mathfrak{s}^2_{\psi}=\frac{2-e^{2\I \psi}-e^{-2\I \psi}}{4}$ and using the expression for the Fourier transform of the Gaussian kernel, we have
    \begin{align}
  4\wh{\Lc^2\vf}_n(\sigma)&=2w_0 \pi^{3/2} e^{-\frac{w_0^2 \sigma^2}{4}}\int_0^{\infty}\bigg[\frac{1}{2}(f_{22}+f_{11})_n(r)(2g_n-g_{n-2}- g_{n+2})(r^{-1}\sigma^{-1})\nonumber \\&\quad +\frac{1}{4}(f_{22}-f_{11}+2\I f_{12})_{n-2}(r)(2g_{n-2}-g_{n-4}- g_{n})(r^{-1}\sigma^{-1})\nonumber \\ &\qquad    + \frac{1}{4}(f_{22}-f_{11}-2\I f_{12})_{n+2}(r)(2g_{n+2}-g_{n}-g_{n+4})(r^{-1}\sigma^{-1})\bigg]\,r^3\,dr.
\end{align}
Taking the Mellin transform on both sides of the above equation, we have
\begin{align}\label{Mellin_L^2f}
    \Nc\bigg[\frac{2e^{\frac{w_0^2}{4\sigma^2}}}{w_0 \pi^{3/2}}\wh{\Lc^2\vf}_n(\sigma^{-1})\bigg](\rho-4)&=\frac{1}{2}\Nc(f_{22}+f_{11})_n(\rho)\Nc(2g_n-g_{n-2}- g_{n+2})(\rho-4)\nonumber \\& \quad +\frac{1}{4}\Nc(f_{22}-f_{11}+2\I f_{12})_{n-2}(\rho)\Nc(2g_{n-2}-g_{n-4}- g_{n})(\rho-4)\nonumber \\ &\quad +  \frac{1}{4}\Nc(f_{22}-f_{11}-2\I f_{12})_{n+2}(\rho) \Nc(2g_{n+2}-g_{n}-g_{n+4})(\rho-4).
\end{align}
From equation \eqref{Mellin_Lf}, we have 
\begin{align}\label{Lf:N(f_11+f_22)_n}
    \Nc(f_{22}+f_{11})_{n}(\rho)&= \frac{2}{\Nc g_n(\rho-2)}\bigg[\Nc\left(\frac{e^{\frac{w_0^2}{4\sigma^2}}}{2w_0 \pi^{3/2}}\wh{\Lc\vf}_n(\sigma^{-1})\right)(\rho-2)\nonumber\\ &-\frac{1}{4}\Nc(f_{22}-f_{11}+2\I f_{12})_{n-2}(\rho)\Nc g_{n-2}(\rho-2)\nonumber\\ &-\frac{1}{4}\Nc(f_{22}-f_{11}-2\I f_{12})_{n+2}(\rho)\Nc g_{n+2}(\rho-2)\bigg].
\end{align}
From equation \eqref{Mellin_L^1f} and using the above expression, we have
\begin{align}\label{Lf+L^1f}
    &4\Nc\bigg[-\I \frac{e^{\frac{w_0^2}{4\sigma^2}}}{w_0\pi^{3/2}}\wh{\Lc^1\vf}_n(\sigma^{-1})\bigg](\rho-3)\Nc g_n(\rho -2)\nonumber \\
     & \qquad - 4\Nc\bigg[\frac{e^{\frac{w_0^2}{4\sigma^2}}}{2w_0 \pi^{3/2}}\wh{\Lc\vf}_n(\sigma^{-1})\bigg](\rho-2) \Nc(g_{n-1}-g_{n+1})(\rho-3)\nonumber \\&= \Nc(f_{22}-f_{11}+2\I f_{12})_{n-2}(\rho)\nonumber \\
    &\qquad \qquad  \times\left[\Nc g_n(\rho-2)\Nc(g_{n-3}-g_{n-1})(\rho-3)- \Nc g_{n-2}(\rho-2) \Nc(g_{n-1}-g_{n+1})(\rho-3)\right] \nonumber \\ &\quad + \Nc(f_{22}-f_{11}-2\I f_{12})_{n
    +2}(\rho)\nonumber \\
    &\qquad \qquad  \times\left[\Nc g_n(\rho-2)\Nc(g_{n+1}-g_{n+3})(\rho-3)- \Nc g_{n+2}(\rho-2) \Nc(g_{n-1}-g_{n+1})(\rho-3)\right].
\end{align}
From \eqref{Mellin_L^2f} and using \eqref{Lf:N(f_11+f_22)_n}, we get
\begin{align}\label{Lf+L^2f}
    &4\Nc\bigg[\frac{2e^{\frac{w_0^2}{4\sigma^2}}}{w_0 \pi^{3/2}}\wh{\Lc^2\vf}_n(\sigma^{-1})\bigg](\rho-4) \Nc g_n(\rho-2) \nonumber \\
    &\qquad -4\Nc\bigg[\frac{e^{\frac{w_0^2}{4\sigma^2}}}{2w_0 \pi^{3/2}}\wh{\Lc\vf}_n(\sigma^{-1})\bigg](\rho-2) \Nc(2g_n-g_{n-2}-g_{n+2})(\rho-4) \nonumber \\&\qquad \quad= \Nc(f_{22}-f_{11}+2\I f_{12})_{n-2}(\rho)\big[ \Nc g_n(\rho-2)\Nc(2g_{n-2}-g_{n-4}-g_n)(\rho-4)\nonumber \\&\qquad \qquad  + \Nc(f_{22}-f_{11}-2\I f_{12})_{n+2}(\rho)\big[\Nc g_n(\rho-2)\Nc(2g_{n+2}-g_n-g_{n+4})(\rho-4) \nonumber\\&\qquad \quad  \qquad - \Nc g_{n-2}(\rho-2)\Nc(2g_n-g_{n-2}-g_{n+2})(\rho-4)\big] \nonumber\\&\qquad \qquad \qquad - \Nc g_{n+2}(\rho-2)\Nc(2g_n-g_{n-2}-g_{n+2})(\rho-4)\big].
\end{align}
Solving the equations \eqref{Lf+L^1f}, \eqref{Lf+L^2f}, we get 
$\Nc(f_{22}-f_{11}+2\I f_{12})_{n-2}(\rho)$ and $ \Nc(f_{22}-f_{11}-2\I f_{12})_{n+2}(\rho)$. Then using the inverse Mellin transform \eqref{def:inverse_mellin}, we have $(f_{22}-f_{11}+2\I f_{12})_{n-2}$ and $(f_{22}-f_{11}-2\I f_{12})_{n+2}$. Thus, from these we can recover $(f_{22}-f_{11})_n$ and $(f_{12})_n$. Next substituting the values of reconstructed $\Nc(f_{22}-f_{11}+2\I f_{12})_{n-2}(\rho)$ and $ \Nc(f_{22}-f_{11}-2\I f_{12})_{n+2}(\rho)$ in \eqref{Lf:N(f_11+f_22)_n}, we can recover $\Nc(f_{11}+f_{22})_n$. Then using the inverse Mellin transform \eqref{def:inverse_mellin}, we can have $(f_{11}+f_{22})_n.$ Hence, we can recover the $n^{th}$ Fourier components of  $f_{11}$ and $f_{22}.$ Therefore, we can recover a symmetric 2-tensor field $\vf$ from  the knowledge of $\Lc\vf, \Lc^1\vf,\Lc^2\vf$. This completes the proof.
\end{proof}
\subsection{Proof of Remark \ref{Rmk:2-tensor_mixed_mom_partial}}
The goal of this subsection is to show that from the knowledge of mixed GbRt of a symmetric $2$-tensor field $\vf$, its first and second moments ($\Mc\vf, \Mc^1\vf, \Mc^2\vf $), we cannot fully recover $\vf$. But we can expect a partial recovery using the current technique. Before proving this, let us first state an operator $A$ introduced by the authors in \cite{mixed_Hoop_saksala_zhai}:
$$A\vf_{i_1j_1}= (-1)^{1-N(j_{1})}\vf_{i_1\delta(j_1)}$$
where $N(j_1)$ represents the number of $1$ present in $j_1$ and $\delta(j_1)$ changes $1$ to $2$ and $2$ to $1.$
Using this operator, we have
\begin{align*}
    Af_{11}=f_{12},\quad Af_{12}= -f_{11}, \quad Af_{21}=f_{22}, \quad Af_{22}=-f_{21}=-f_{12}.
\end{align*}
Let us denote 
\begin{align}
    \tilde{\vf}= \mbox{Sym}(A\vf)=\begin{bmatrix}
        f_{12} & \frac{f_{22}-f_{11}}{2}\\
        \frac{f_{22}-f_{11}}{2} & -f_{12}
    \end{bmatrix}.
\end{align}
Note that $\Mc\vf=-\Lc\tilde{\vf}$. Also, $\Mc^1\vf=- \Lc^1\tilde{\vf}$ and $\Mc^2\vf=- \Lc^2\tilde{\vf}.$ 
Using the similar procedure as in the subsection \ref{Sec: Rec_Lf,L^1f,L^2f}, we can only expect to recover the $n^{th}$ Fourier components $(f_{12})_n$ and $(f_{22}-f_{11})_n$. So, from the combination of $\Mc\vf, \Mc^1\vf$ and $\Mc^2\vf$ we can recover $\vf$ partially. For full reconstruction, we need to consider more data, for example $\Lc\vf$ or $\Tc\vf$, etc.
\section{Acknowledgements}\label{Acknowledgements}
SR was supported by the US National Science Foundation Grant Number 2309491.
\bibliography{refs}
\bibliographystyle{plain}

\end{document}